\documentclass[
onecolumn, 
hidempi, 
]{mpi2015-cscpreprint}

\RequirePackage{fix-cm}

\usepackage{graphicx}
\usepackage{mathptmx}      
\usepackage{amsmath,amsthm}
\usepackage{amsfonts}
\usepackage{cite}
\usepackage{epstopdf}
\usepackage{algorithm}
\usepackage{algorithmic}
\usepackage{multirow}
\usepackage{pifont}
\usepackage{underscore}
\newcommand{\cmark}{\ding{51}}
\newcommand{\xmark}{\ding{55}}
\newcommand{\rdeg}{\textup{rdeg}}
\DeclareMathOperator*{\argmax}{arg\,max}
\usepackage{xcolor}
\newcommand{\rev}[1]{{\color{black}#1}}

\newtheorem{theorem}{Theorem}[section]

\newtheorem{lemma}[theorem]{Lemma}
\newtheorem{corollary}[theorem]{Corollary}
\newtheorem{remark}{Remark}[section]

\begin{document}
	
	\title{Barycentric rational approximation for learning the index of a dynamical system from limited data}
	
	
	\author[$\ast$]{Davide Pradovera}
	\affil[$\ast$]{ Stockholm University, Albanov\"agen 28, 11419, Stockholm, Sweden.\authorcr
\email{davide.pradovera@math.su.se}, \orcid{0000-0003-0398-1580}}

\author[$\dag$]{Ion Victor Gosea}
\affil[$\dag$]{Max Planck Institute for Dynamics of Complex Technical Systems,
Sandtorstr. 1, 39106 Magdeburg, Germany.\authorcr
\email{gosea@mpi-magdeburg.mpg.de}, \orcid{0000-0003-3580-4116}}

\author[$\ddagger$]{Jan Heiland}
\affil[$\ddagger$]{Ilmenau University of Technology, Weimarer Str. 25, 98693 Ilmenau, Germany and Max Planck Institute for Dynamics of Complex Technical Systems,
Sandtorstr. 1, 39106 Magdeburg, Germany \authorcr
\email{jan.heiland@tu-ilmenau.de}, \orcid{0000-0003-0228-8522}}

\abstract{We consider the task of data-driven identification of dynamical systems, specifically for systems whose behavior at large frequencies is non-standard, as encoded by a non-trivial \emph{relative degree} of the transfer function or, alternatively, a non-trivial \emph{index} of a corresponding realization as a descriptor system. We develop novel surrogate modeling strategies that allow state-of-the-art rational approximation algorithms (e.g., AAA and vector fitting) to better handle data coming from such systems with non-trivial relative degree. Our contribution is twofold. On one hand, we describe a strategy to build rational surrogate models with \emph{prescribed relative degree}, with the objective of mirroring the high-frequency behavior of the high-fidelity problem, when known. The surrogate model's desired degree is achieved through constraints on its barycentric coefficients, rather than through ad-hoc modifications of the rational form. On the other hand, we present a degree-identification routine that allows one to estimate the \emph{unknown relative degree} of a system from low-frequency data. By identifying the degree of the system that generated the data, we can build a surrogate model that, in addition to matching the data well (at low frequencies), has enhanced extrapolation capabilities (at high frequencies). We showcase the effectiveness and robustness of the newly proposed method through a suite of numerical tests.}


\novelty{In this work, state-of-the-art rational-approximation algorithms are
tailored to the approximation of classes of dynamical systems, whose behavior at
large frequencies is not proper, e.g., systems modeled as descriptor
systems through differential-algebraic equations. We develop strategies to enable the construction of surrogate models with given relative degree, without destroying the standard barycentric form. We also describe how to automatically detect the relative degree of a system from low-frequency measurements of its frequency response.}

\maketitle

\section{Introduction}\label{sec:intro}

The approximation of large-scale dynamical systems \cite{ACA05} is crucial for achieving various goals such as employing efficient simulation or devising robust, automatic control laws. Model order reduction (MOR) is a collection of techniques for reducing the computational complexity of mathematical models in numerical simulations. With the ever-developing technology in many engineering fields, more and more complex mathematical models need to be numerically simulated to get deeper insights into the physics of many applied problems. In this framework, MOR aims to construct cheap and fast, but also reliable and accurate surrogate models for the original complex mathematical problem (described by convoluted, coupled, or highly nonlinear dynamical laws). Historically, MOR methods were typically intrusive, as explicit access to the latter models was typically required to compute the reduced-order counterparts. However, in the last decade, increased effort has been allocated to the computation of surrogate models in a non-intrusive fashion, i.e., by using data-driven approaches, rather than projection-based ones.

Data-driven (non-intrusive) techniques represent a viable alternative to classical (intrusive) methods of MOR, which rely on explicit access to the large-scale model. Unlike intrusive methods, data-driven surrogate models do not require explicit knowledge of the problem structure or matrices. Instead, low-order models can be directly computed from data, such as snapshots of the system's state-space evolution, alongside the control inputs and, optionally, the observed outputs.

Frequency domain analysis methods are used to approximate and analyze the behavior of a dynamical system using its transfer function, a system-invariant quantity that does not depend on the system's states or variables, but only on the input-output map that encodes the ``true dimensionality'' of the problem. For finite-dimensional linear time-invariant systems, the transfer function is a rational function, whereas infinite-dimensional systems (e.g., delayed or integro-differential ones) lead to more complex irrational structures. By approximating the transfer function of a large-scale system, one can provide insight on the system's response to a specific range of frequencies actuated by the control input. The class of methods that are aiming at this can be thought of as rational approximation tools that provide various techniques for an accurate data fit. A standard way of achieving this is by means of \textit{interpolation}, although one must move beyond ill-conditioned classical methods involving polynomials. Instead, through \emph{rational interpolation}, especially when using the \emph{barycentric} form, one can recover both approximation quality and good conditioning \cite{berrut2004barycentric}.

One notable approach for rational interpolation is based on Loewner matrices, was originally presented in \cite{AA86} and will be referred to as the Antoulas-Anderson (AA) method. A rational interpolant in barycentric form is built by computing the null space of a Loewner matrix based on the available data. Another method is the vector fitting (VF) algorithm \cite{VF}, which uses a linearized least-squares fitting approach. This algorithm iteratively adjusts some parameters of the rational functions (namely, the support points, as defined below) to minimize the mismatch between the model and the actual data, resulting in an accurate recovery of the transfer function. On the other hand, the adaptive AA (AAA) algorithm \cite{NST18} combines elements of both approaches, blending interpolation and least-squares fitting. It aims at finding a rational approximant by iteratively adjusting the model based on greedily selected interpolation points and on a least-squares fit.

\subsection{The barycentric form}

What all the above methods have in common is the barycentric form used to express the rational approximation. It is an alternative to other representations of transfer functions, e.g., as the ratio of two polynomials or, as in the ``Heaviside'' pole-residue format, as a sum of simple fractions. The barycentric form is a numerically stable representation \cite{berrut2004barycentric} of the rational approximant to be computed. It represents a rational function as a ratio of two sums of fractions with identical singularities (or poles):
\begin{equation}\label{eq:baryfull}
r(s) = \sum_{k=0}^m \frac{n_k}{s - s_k} \Bigg/ \sum_{k=0}^m\frac{d_k}{s - s_k}.
\end{equation}
The poles of both numerator and denominators, namely, the set $\{s_k\}_{k=0}^m$, are referred to as \emph{support points}. The numerators are different, containing the so-called \emph{barycentric weights} $n_k$ and $d_k$. One of the most useful features of the barycentric form is that one can encode \emph{interpolation} properties in its structure, in a numerically stable way. Specifically, the value of $r$ at each support point $s_k$ equals $n_k/d_k$ by design. As such, enforcing an arbitrary value $f_k$ at any support point is as easy as setting the numerator barycentric coefficient $n_k$ as $n_k:=d_kf_k$. If this is done at all support points, one obtains the \emph{interpolatory} barycentric form
\begin{equation}\label{eq:bary}
r(s) = \sum_{k=0}^m \frac{w_k f_k}{s - s_k} \Bigg/ \sum_{k=0}^m\frac{w_k}{s - s_k},
\end{equation}
which attains $r(s_k)=f_k$ for all $k$. Note that, for historical reasons, we have switched to the symbol $w_k$ to denote the denominator barycentric weights.

It is well known that $r$ in either barycentric form \eqref{eq:baryfull} or \eqref{eq:bary} is a rational function of \emph{rational type} $(m,m)$, i.e., $r$ may be expressed as the ratio of two complex-valued polynomials, each of degree at most $m$. In order to obtain this alternative form, it suffices to multiply both numerator and denominator by the nodal degree-$(m+1)$ polynomial $\pi(s)=\prod_{k=0}^m(s-s_k)$. This leads to the ``rational'' (as opposed to ``barycentric'') form of $r$:
\begin{equation}\label{eq:baryfirst}
r(s) = \sum_{k=0}^mn_k L_k(s) \Bigg/ \sum_{k=0}^md_k L_k(s)\;\text{or}\; r(s) = \sum_{k=0}^mw_k f_k L_k(s) \Bigg/ \sum_{k=0}^mw_k L_k(s),
\end{equation}
for \eqref{eq:baryfull} and \eqref{eq:bary}, respectively. Above, $L_k(s)=\pi(s)/(s-s_k)$ is a degree-$m$ ``non-nor\-mal\-ized'' Lagrange polynomial. The rational form of $r$ is often avoided due to how unstable the evaluation of the polynomials $L_k$ generally is. On the other hand, the barycentric form is preferred due to its more beneficial numerical properties.

\subsection{Motivation: DAE index and relative degree}

In this work, we will show how, by imposing extra conditions on the barycentric weights appearing in \eqref{eq:baryfull} and \eqref{eq:bary}, one can achieve a prescribed rational type. To be more precise, we will focus on prescribing a certain \emph{(relative) degree} for the rational function, defined as the difference between the degrees of numerator and denominator in polynomial form.

This quantity is of interest, e.g., in the study of descriptor systems characterized by differential-algebraic equations (DAEs). Indeed, while a standard ODE system (without a \emph{feedthrough term}) has a transfer function with relative degree at most $-1$, the transfer function of a linear DAE can contain terms of higher relative degree, often referred to as \emph{polynomial} or \emph{improper part}, dominating the response at high frequencies. We refer the reader to \cite{petzold1982differential,KM2006differential} for more in-depth characterization, properties, and solutions of DAE systems. In the literature of DAEs, a concept closely related to relative degree is that of \emph{index}. Although various different concepts of index exist, we will casually use the term ``DAE index'' to denote the relative degree plus one. This is correct if one refers to the \emph{Kronecker} index, assuming it to be defined and that there is no realization of smaller index. Concerning the ultimate motivation for our work, we note that the knowledge of the DAE index of a system is fundamental for controller design and MOR \cite{BerIR23,morBenS17}.

In recent years, there have been attempts to estimate the DAE index from frequency response data. In \cite{morGosZA20,morAntGH20}, the coefficients of the polynomial part of the transfer function are estimated from samples at high frequency. Moreover, a method that is tailored for fitting DAEs of index $2$ (i.e., relative degree $1$) was recently proposed in \cite{GosH23}, extending the AA procedure. There, the standard barycentric form was modified, adding an extra weight to the barycentric numerator, as a way to account for the polynomial part of the transfer function.

However, all these methods are characterized by two main shortcomings: (i) the target relative degree must be set in advance, requiring an expert opinion on (an upper bound for) the DAE index, and (ii) they either require modifying the barycentric form or estimating the polynomial coefficients separately. Also, we note that these methods are tailored to positive relative degree, corresponding to the class of DAE systems.

In this paper, we explicitly address the aforementioned shortcomings and aim at developing a fully automatic method that uses the standard, unaltered barycentric form to estimate the relative degree of a system's transfer function given only limited data, without any prior information on the degree. Notably, such data is in the form of transfer function samples, which are \emph{not} assumed to be at high-frequency values. For the sake of versatility, we mainly follow the AAA algorithm, which chooses the interpolation points adaptively. In this sense, our contribution may also be viewed as an extension of the AAA algorithm, although our presentation is also relevant for other rational approximation techniques. Concretely, we discuss how to apply our method in combination with least-squares-based rational \emph{approximation}, which may achieve substantially higher robustness than \emph{interpolation}, e.g., when only noisy samples of the transfer function are available. The method proposed here can be applied to both positive and negative relative degrees, regardless of how large or small they are.

Our method relies on techniques for prescribing a given degree within rational-function-based model reduction of frequency-domain data. While the algebraic characterization of nonzero relative degree is not entirely new in rational approximation (see, e.g., \cite{filip_rational_2018}), one of our main contributions lies in translating these insights into a stable and practically viable algorithmic framework that operates robustly even for large relative degrees, which were previously numerically inaccessible. Notably, our framework achieves higher numerical robustness than prior works by incorporating novel stable procedures for rational extrapolation, which play a crucial role in reliably predicting a system's behavior outside the sampled frequency range.

The paper is structured as follows: Section~\ref{sec:RatApprox} reviews rational approximation, including AAA as an exemplifying framework. In Section~\ref{sec:RatDegree}, we rigorously derive algebraic conditions that relate a rational function's relative degree to its barycentric weights. Such theoretical results are first used in Section~\ref{sec:RatApproxConstrained} to devise an algorithm for building a rational approximant of prescribed degree, and then in Section~\ref{sec:DegreeAutomatic} to develop a procedure for data-driven identification of a dynamical system's relative degree. A robust least-squares extension of our method is presented in Section~\ref{sec:Extensions}. Our novel methods are then numerically tested in Section~\ref{sec:NumericalExamples} on a variety of numerical examples, including several MOR benchmarks.

\section{Rational approximation using the barycentric form}\label{sec:RatApprox}

In this section, we summarize established rational approximation methods that have also been used as model reduction tools throughout the years.

Among the above-mentioned methods, both AA and AAA are interpolatory, i.e., they use form \eqref{eq:bary}, with interpolation being enforced at all support points. As a consequence, the support points must be chosen as a subset of the sample points. The main difference between AA and AAA is that AA fixes the support points in advance, while AAA selects them adaptively through an iterative procedure.

On the other hand, VF is not interpolatory and thus makes use of the general barycentric form \eqref{eq:baryfull}. The support points are iteratively updated, through the so-called Sanathanan-Koerner (SK) iterations, with the objective of finding their ``optimal'' location. The number of support points stays constant and must be fixed in advance, as opposed to AAA, where it gradually increases.

In our discussion, we look at all the above methods as approaches for rational approximation (or ``surrogate modeling'') of a scalar function $f:\mathbb{C}\to\mathbb{C}$, representing, in our application, the transfer function of a large-scale, complex dynamical system. The to-be-approximated function $f$ is sampled at a set of $m'>0$ points $\{s_j'\}_{j=1}^{m'}$. These are typically located on the imaginary axis $\textrm{i}\mathbb{R}$ whenever the task involves the identification of dynamical systems from frequency-domain measurements.

Within each iteration of AAA and VF, an optimization problem is used to characterize the barycentric weights of the rational approximant. In its more general formulation, this problem is the same one that must be solved (only once) in AA as well. Specifically, one tries to find the barycentric weights that minimize the \emph{$\ell^2$ approximation error} at the sample points, namely,
\begin{equation*}
e:=\left(\sum_{j=1}^{m'}\left|f(s_j')-\sum_{k=0}^m\frac{n_k}{s_j'-s_k} \Bigg/ \sum_{k=0}^m\frac{d_k}{s_j'-s_k}\right|^2\right)^{1/2}.
\end{equation*}
Due to the nonlinearity of this target function (with respect to $d_k$), a linearized version of the problem is considered: multiplying by the surrogate denominator leads to
\begin{equation}\label{eq:LSVF}
\begin{aligned}
	\min_{n_0,\ldots,n_m,d_0,\ldots,d_m\in\mathbb{C}} & \sum_{j=1}^{m'}\lambda_j\left|\sum_{k=0}^m \frac{f(s_j')d_k-n_k}{s_j'-s_k}\right|^2,\\
	\textrm{such that} \quad & \sum_{k=0}^m|d_k|^2=1.
\end{aligned}
\end{equation}
Note the addition of a normalization constraint, to avoid the trivial solution $n_k=d_k=0$ for all $k$, as well as the presence of the weights $\lambda_j$, which are equal to 1 in AA and AAA, while they vary throughout the SK iterations in VF.

In the interpolatory approaches (AA and AAA), the imposition of interpolation conditions reduces the number of degrees of freedom, eliminating the numerator weights, cf.~\eqref{eq:bary}. The corresponding optimization problem reads
\begin{equation}\label{eq:LSAA}
\begin{aligned}
	\min_{w_0,\ldots,w_m\in\mathbb{C}} & \sum_{j=1}^{m'}\left|\sum_{k=0}^m \frac{f(s_j')-f_k}{s_j'-s_k}w_k\right|^2,\\
	\textrm{such that} \quad & \sum_{k=0}^m|w_k|^2=1.
\end{aligned}
\end{equation}
Both \eqref{eq:LSVF} and \eqref{eq:LSAA} can be solved by standard techniques from computational linear algebra, e.g., through singular value decomposition (SVD), cf.~Section~\ref{sec:RatApproxConstrained}.

As mentioned above, we will focus on the AAA algorithm in our upcoming presentation, although generalizations to other techniques like VF are possible, as described in Section~\ref{sec:Extensions}. Our choice is due to AAA's great effectiveness and flexibility. Indeed, the AAA algorithm, originally proposed in \cite{NST18}, has since been extended and developed in recent years, including applications to nonlinear eigenvalue problems \cite{lietaert2022automatic,guttel2022robust,pradovera22,guttel2022randomized}, MOR of parameterized linear dynamical systems \cite{nobile2021,CBG2023}, MOR of linear systems with quadratic outputs \cite{gosea2022data}, rational approximation of periodic functions \cite{baddoo2021aaatrig}, and rational approximation of matrix-valued functions \cite{elsworth2019conversions,GG20,morAumBGetal23,osti2005602}.

The characterizing feature of the AAA algorithm is its greedy choice of support points. These are incrementally selected through an iterative procedure, terminating as soon as the $\ell^\infty$ approximation error, namely, $\max_{j=1,\ldots,m'}|r(s_j')-f(s_j')|$, is below a user-defined tolerance $\varepsilon_{\text{AAA}}>0$. A pseudo-code formulation of the AAA algorithm is provided in a later section, namely, Algorithm~\ref{algo:aaa} with $\delta=0$.

\section{Relative degree of barycentric forms}\label{sec:RatDegree}
We investigate here how a rational function's degree can be found, based on the function's barycentric coefficients.

We begin by pointing out that the rational type $(m,m)$ of the rational function $r$ in \eqref{eq:baryfull} or \eqref{eq:bary} may not be ``exact'', i.e., the degrees of the numerator and denominator appearing in the rational form of $r$ may be smaller than $m$. Following this observation, we rigorously define the relative degree $\rdeg(r)$ of $r$ as the difference between the exact degrees of numerator and denominator of $r$ in the form \eqref{eq:baryfirst}. An equivalent definition involves an asymptotic analysis:
\begin{equation*}
r(s)=\mathcal{O}\left(s^{\rdeg(r)}\right)\quad\text{as }s\to\infty.
\end{equation*}

\begin{remark}\label{rem:degreesystem}
In systems theory, one defines the relative degree of a \emph{system} (as opposed to that of a rational function) by an abuse of notation, as the relative degree of its transfer function (see, e.g., \cite[Ch. 4.1]{Isi85}). Generalizations for nonlinear systems are also possible.
\end{remark}

We recall that one of our ultimate objectives is to compute rational functions of given relative degree, in order to approximate well the high-fidelity model's behavior for large frequencies. Accordingly, it is crucial to be able to relate the relative degree $\rdeg(r)$ to the support values $f_0,\ldots,f_m$ and to the barycentric weights $w_0,\ldots,w_m$. To this aim, we first derive a useful technical identity.

\begin{lemma}\label{lem:eqasympt}
Let $\{\alpha_k\}_{k=0}^m\cup\{s_k\}_{k=0}^m\subset\mathbb C$. If $|s|>\max_{k=0,\ldots,m}|s_k|$, we have the series representation
\begin{equation}\label{eq:eqasympt}
	\sum_{k=0}^m\frac{\alpha_k}{s - s_k}=\sum_{l=0}^\infty\left(\sum_{k=0}^m\alpha_ks_k^l\right)\frac1{s^{l+1}}.
\end{equation}
\end{lemma}
\begin{proof}
The claim follows by the Laurent expansion of the geometric series:
\begin{equation*}
	\frac1s\sum_{k=0}^m\frac{\alpha_k}{1-s_k/s}=\frac1s\sum_{k=0}^m\alpha_k\sum_{l=0}^\infty\left(\frac{s_k}{s}\right)^l=\sum_{l=0}^\infty\left(\sum_{k=0}^m\alpha_ks_k^l\right)\frac1{s^{l+1}}.
\end{equation*}
The requirement that $|s|>|s_k|$ for all $k=0,\ldots,m$ is necessary for the (absolute) convergence of all $(m+1)$ geometric series.
\end{proof}

We are now ready to state our main results, starting from the following theorem. We mention that, during the writing of this work, we became aware of a similar 25-year-old result appearing in \cite{berrut1997barycentric}, which was also recently rediscovered in \cite{filip_rational_2018}.

\begin{theorem}\label{th:degree}
Let $\mu$ and $\nu$ be non-negative integers $\leq m$ such that
\begin{equation}\label{eq:numdegree}
	\begin{cases}
		\sum_{k=0}^mw_kf_ks_k^l=0\quad\text{for } l=0,1,\ldots,\mu-1,\\
		\sum_{k=0}^mw_kf_ks_k^\mu\neq 0,
	\end{cases}
\end{equation}
and
\begin{equation}\label{eq:dendegree}
	\begin{cases}
		\sum_{k=0}^mw_ks_k^l=0\quad\text{for } l=0,1,\ldots,\nu-1,\\
		\sum_{k=0}^mw_ks_k^\nu\neq 0.
	\end{cases}
\end{equation}
Then the rational function $r$ in \eqref{eq:bary} has exact type $(m-\mu,m-\nu)$, and its relative degree is $\rdeg(r)=\nu-\mu$.

Except for trivial cases ($w_kf_k=0$ for all $k$ or $w_k=0$ for all $k$), the converse also holds true: if the exact type of $r$ is $(m-\mu,m-\nu)$, then \eqref{eq:numdegree} and \eqref{eq:dendegree} hold.
\end{theorem}

\begin{proof}
We rely on the representation of $r$ in the rational form \eqref{eq:baryfirst}. For the sake of brevity, we will only prove results for the denominator
\begin{equation*}
	q(s)=\sum_{k=0}^mw_kL_k(s)=\pi(s)\sum_{k=0}^m\frac{w_k}{s - s_k}.
\end{equation*}
The equivalent results concerning the numerator $p(s)=\sum_{k=0}^mw_kf_kL_k(s)$ follow by the same arguments, by simply replacing any instance of ``$w_k$'' with ``$w_kf_k$''.

Consider
\begin{equation*}
	d(s)=\frac{q(s)}{\pi(s)}=\sum_{k=0}^m\frac{w_k}{s - s_k}.
\end{equation*}
By Lemma~\ref{lem:eqasympt},
\begin{equation*}
	d(s)=\sum_{l=0}^\infty\left(\sum_{k=0}^mw_ks_k^l\right)\frac1{s^{l+1}}
\end{equation*}
as $s\to\infty$. On the other hand, the nodal polynomial $\pi$ satisfies
\begin{equation*}
	\pi(s)=\prod_{k=0}^m(s-s_k)=s^{m+1}+\mathcal{O}\left(s^m\right).
\end{equation*}
Putting these two observations together leads to
\begin{equation}\label{eq:laurent}
	q(s)=\sum_{l=0}^\infty\left(\sum_{k=0}^mw_ks_k^l\right)\left(s^{m-l}+\mathcal{O}\left(s^{m-l-1}\right)\right).
\end{equation}

The claims of Theorem~\ref{th:degree} follow from here. Indeed, if \eqref{eq:dendegree} is true, then the first $\nu$ terms in \eqref{eq:laurent} disappear while the $(\nu+1)$-th remains:
\begin{equation*}
	q(s)=\left(\sum_{k=0}^mw_ks_k^\nu\right)\left(s^{m-\nu}+\mathcal{O}\left(s^{m-\nu-1}\right)\right)+\underbrace{\sum_{l=\nu+1}^\infty\left(\sum_{k=0}^mw_ks_k^l\right)\left(s^{m-l}+\mathcal{O}\left(s^{m-l-1}\right)\right)}_{=\mathcal{O}\left(s^{m-\nu-1}\right)},
\end{equation*}
and $q(s)$ has leading term $\left(\sum_{k=0}^mw_ks_k^\nu\right)s^{m-\nu}$, which is nonzero by \eqref{eq:dendegree}.

On the other hand, if a non-identically-zero $q$ has degree exactly $m-\nu$, then:
\begin{itemize}
	\item the first $\nu$ terms in \eqref{eq:laurent} must disappear, since they have order $>m-\nu$; as such, $\sum_{k=0}^mw_ks_k^l=0$ for $l=0,\ldots,\nu-1$;
	\item the $(\nu+1)$-th term in \eqref{eq:laurent} cannot disappear, since all the terms from the $(\nu+2)$-th onward have order $<m-\nu$; as such, $\sum_{k=0}^mw_ks_k^\nu\neq0$.
\end{itemize}
\end{proof}

Note that the hypotheses of Theorem~\ref{th:degree} can be equivalently stated in terms of (generalized) Vandermonde matrices, on which we will rely for building our surrogate model.
\begin{corollary}\label{th:vandermonde}
Let $\phi_0(s),\ldots,\phi_{\mu-1}(s)$ be any basis of $\mathbb{P}_{\mu-1}(\mathbb{C})$, the space of polynomials of degree $<\mu$. Define $F=\textup{diag}(f_0,\ldots,f_m)\in\mathbb{C}^{(m+1)\times(m+1)}$ and let
\begin{equation}\label{eq:vander}
	V_{\mu-1}=\begin{bmatrix}
		\phi_0(s_0) & \cdots & \phi_{\mu-1}(s_0) \\
		\vdots & \ddots & \vdots \\
		\phi_0(s_m) & \cdots & \phi_{\mu-1}(s_m) \\
	\end{bmatrix}\in\mathbb{C}^{(m+1)\times\mu}
\end{equation}
be the corresponding generalized Vandermonde matrix. If $\mathbf{w}=[w_0,\ldots,w_m]^\top$ is non\-zero and lies in the null space of $V_{\mu-1}^\top F$, then the exact degree of the numerator of $r$ in \eqref{eq:baryfirst} is at most $m-\mu$.

Moreover, replacing $\mu$ with $\nu$ and $F$ with the identity matrix in the previous statement leads to a similar degree condition on the denominator of $r$: if $\mathbf{w}$ is nonzero and lies in the null space of $V_{\mu-1}^\top$, then the exact degree of the denominator of $r$ in \eqref{eq:baryfirst} is at most $m-\nu$.
\end{corollary}

\begin{proof}
Define $\mathbf{n}=F\mathbf{w}=[w_0f_0,\ldots,w_mf_m]^\top$, so that, in particular, $V_{\mu-1}^\top \mathbf{n}=V_{\mu-1}^\top F\mathbf{w}$. It is easily seen that the equality conditions in \eqref{eq:numdegree} and \eqref{eq:dendegree} correspond to null-space constraints for $\mathbf{n}$ and $\mathbf{w}$, respectively, involving Vandermonde matrices induced by the monomial basis $\phi_l(s)=s^l$, up to degrees $\mu-1$ and $\nu-1$, respectively. Since the constraints in \eqref{eq:numdegree} and \eqref{eq:dendegree} are linear in $s_k^l$ for all $k$ and $l$, the claim follows by a simple change-of-basis argument (for the polynomial basis).
\end{proof}

\section{Rational approximation with given relative degree}\label{sec:RatApproxConstrained}

The theoretical results from the previous section give us a tool for checking the relative degree of a rational function by evaluating linear combinations of its barycentric coefficients. While this can be useful on its own, we can leverage the fact that Theorem~\ref{th:degree} is a co-implication to design ways of \emph{enforcing} a given degree within rational approximation.

We restrict our attention here to AAA, although the considerations that follow can be generalized to many other rational approximation algorithms, as long as they rely on the barycentric form, cf.~Section~\ref{sec:Extensions}. The plain version of AAA is not concerned with the degree of the rational interpolant $r$ that the algorithm builds. As such, the resulting rational function $r$ will usually have exact type $(m,m)$. Indeed, round-off errors will generally make both $\sum_{k=0}^mw_kf_k$ and $\sum_{k=0}^mw_k$ nonzero, so that Theorem~\ref{th:degree} implies the largest possible type, yielding degree $0$.

Relying on the results from Section~\ref{sec:RatDegree}, we can prescribe an arbitrary type for the rational approximant by simply adding linear constraints on its barycentric coefficients \emph{during} their computation. For instance, imposing a positive degree $\delta>0$ leads to
\begin{equation}\label{eq:LSAAConstrained}
\begin{aligned}
	\min_{w_0,\ldots,w_m\in\mathbb{C}} & \sum_{j=1}^{m'}\left|\sum_{k=0}^m \frac{f(s_j')-f_k}{s_j'-s_k}w_k\right|^2,\\
	\textrm{such that} \quad & \sum_{k=0}^m|w_k|^2=1\\
	\textrm{and} \quad & \sum_{k=0}^mw_ks_k^l=0\quad\text{for } l=0,1,\ldots,\delta-1.
\end{aligned}
\end{equation}
(Enforcing a degree $\delta<0$ requires constraints involving $w_kf_k$ instead of just $w_k$.)

\begin{remark}
It is crucial to observe that the solution of \eqref{eq:LSAAConstrained} is not \emph{guaranteed} to have exact relative degree $\delta$. Indeed,
\begin{itemize}
	\item the numerator of $r$ might have degree lower than $m$ (if $\sum_{k=0}^mw_kf_k=0$);
	\item the denominator of $r$ might have degree lower than $m-\delta$ (if $\sum_{k=0}^mw_ks_k^\delta=0$).
\end{itemize}
(Similar considerations also apply for negative degrees.) However, these are not serious issues, since both above conditions are (i) extremely unlikely due to round-off and (ii) easy to spot. In practice, one might want to return an error message if either $\sum_{k=0}^mw_kf_k$ or $\sum_{k=0}^mw_ks_k^\delta$ is smaller than some tolerance, say $10^{-15}$.
\end{remark}

Problem \eqref{eq:LSAAConstrained} can easily be solved by extending ideas from \cite[Section 6.2.3]{Golub1996}. For instance, following Theorem~\ref{th:vandermonde}, one can recast \eqref{eq:LSAAConstrained} by restricting \eqref{eq:LSAA} to the null space of the (generalized) Vandermonde matrix $V_{|\delta|-1}$, as follows.

First, one computes a matrix $Q\in\mathbb{C}^{(m+1)\times(m+1-|\delta|)}$ with orthonormal columns that span the null space of $V_{|\delta|-1}$. (Note that, if the target degree $\delta$ is negative, one must left-multiply the Vandermonde matrix by the diagonal matrix $F$ defined in Theorem~\ref{th:vandermonde}.) To this aim, it suffices, e.g., to extract the rightmost block of the orthogonal matrix resulting from a full QR factorization of $V_{|\delta|-1}$. Alternatively, an Arnoldi-type procedure, see, e.g., \cite{brubeck_vandermonde_2021}, may achieve higher robustness, potentially enabling stable computations even when the defect $|\delta|$ is quite large.

Then, one solves the following size-$(m+1-|\delta|)$ problem \emph{without linear constraints}:
\begin{equation*}
\begin{aligned}
	\min_{v_0,\ldots,v_{m-|\delta|}\in\mathbb{C}} & \sum_{j=1}^{m'}\left|\sum_{k=0}^m\frac{f(s_j')-f_k}{s_j'-s_k}\sum_{l=0}^{m-|\delta|}Q_{kl}v_l\right|^2,\\
	\textrm{such that} \quad & \sum_{l=0}^{m-|\delta|}|v_k|^2=1.
\end{aligned}
\end{equation*}
A solution can be found, e.g., by computing a minimal right singular vector of the matrix $LQ$, with $L$ being the tall Loewner matrix
\begin{equation}\label{eq:loewnermatrix}
L=\begin{bmatrix}
	\frac{f(s_1')-f_0}{s_1'-s_0} & \cdots & \frac{f(s_1')-f_m}{s_1'-s_m}\\
	\vdots & \ddots & \vdots\\
	\frac{f(s_{m'}')-f_0}{s_{m'}'-s_0} & \cdots & \frac{f(s_{m'}')-f_m}{s_{m'}'-s_m}
\end{bmatrix}\in\mathbb{C}^{m'\times(m+1)}.
\end{equation}
The desired barycentric coefficients are given by a simple (orthogonal) change of basis: $w_k=\sum_{l=0}^{m-|\delta|}Q_{kl}v_l$.

A pseudo-code summarizing the procedure is provided in Algorithm~\ref{algo:aaa}. Different from the original formulation of AAA \cite{NST18}, we employ the relative error, as opposed to the absolute one, to drive the adaptive selection of the support points and to terminate the sampling loop based on the tolerance $\varepsilon_{\text{AAA}}$. Our aim in doing this is to achieve a higher robustness in the approximation of general transfer functions, no matter the underlying system. Indeed, relying on the absolute error behooves one to scale the tolerance on a case-by-case basis, as a way to account for the absolute magnitude of the signal. Another advantage of using the relative error is more specific to our framework: the relative error is more appropriate when the magnitude of the signal displays large variations, e.g., as a result of a nonzero system degree.

\begin{algorithm}[ht]
\caption{Data-driven AAA algorithm for rational interpolation (relative error version)}
\label{algo:aaa}
\begin{algorithmic}
	\REQUIRE{data points $\{(s_j',f(s_j'))\}_j$, AAA tolerance $\varepsilon_{\text{AAA}}>0$, target degree $\delta\in\mathbb{N}$}
	\STATE{Initialize $\Lambda:=\{s_j'\}_j$ and $r:\equiv\frac1{m'}\sum_{j=1}^{m'}f(s_j')$}
	\FOR{$m=0,1,...$}
	\STATE{Choose $s_m:=\argmax_{s'\in\Lambda}|f(s')-r(s')|/|f(s')|$}
	\STATE{Set $f_m:=f(s_m)$ and $\Lambda:=\Lambda\setminus\{s_m\}$}
	\IF{$|f_m-r(s_m)|/|f_m|\leq\varepsilon_{\text{AAA}}$}
	\RETURN{$r$}
	\ENDIF
	\STATE{$\Delta:=\min\{|\delta|,m\}$}
	\IF{$\Delta\neq 0$}
	\STATE{Assemble the (generalized) Vandermonde matrix $V_{\Delta-1}$ as in \eqref{eq:vander}}
	\IF{$\delta<0$}
	\STATE{Assemble $F$ as in Theorem~\ref{th:vandermonde} and update $V_{\Delta-1}:=FV_{\Delta-1}$}
	\ENDIF
	\STATE{Compute a full QR decomposition $\widetilde{Q}R:=V_{\Delta-1}$}
	\STATE{Define $Q$ as the $(m+1-\Delta)$ rightmost columns of $\widetilde{Q}$}
	\ELSE
	\STATE{Define $Q$ as the identity matrix of size $(m+1)$}
	\ENDIF
	\STATE{Assemble the Loewner matrix $L$ in \eqref{eq:loewnermatrix} and compute the SVD $U\Sigma V^H:=LQ$}
	\STATE{Extract the barycentric coefficients $(w_k)_{k=0}^m$ from the last column of $QV$}
	\STATE{Define the rational function $r$ as in \eqref{eq:bary}}
	\ENDFOR
\end{algorithmic}
\end{algorithm}

\begin{remark}\label{rem:deg}
At the $m$-th iteration, Algorithm~\ref{algo:aaa} imposes an adjusted degree equal to $\textup{sign}(\delta)\min\{|\delta|,m\}$, which, for small values of $m$, may be less (in magnitude) than the target $\delta$. This is out of necessity, since the degree of a rational function cannot be larger than its type.
\end{remark}

\subsection{Stable extrapolation}\label{subsec:asympt}
By following the strategy described in the previous section, one can build a rational approximation $r$ with an arbitrary relative degree. The surrogate model $r$ can be efficiently and stably evaluated through the barycentric form \eqref{eq:bary}. However, whenever $\rdeg(r)\neq0$, this stability is guaranteed only at frequencies $s$ located near the sample points. Indeed, if $r$ has a nonzero degree, evaluating it at large frequencies may lead to serious numerical cancellation effects. As detailed in Lemma~\ref{lem:asympt} below, the ultimate reason for such cancellation is given by identities \eqref{eq:numdegree} and \eqref{eq:dendegree} in Theorem~\ref{th:degree}, which state that certain linear combinations of the numerator or denominator coefficients equal 0. This is also showcased in Section~\ref{sec:NumericalExamples}.

This is a serious issue, considering that rational surrogates are often used for forecasting and prediction of systems' frequency responses outside the sampled frequencies. This is even more the case in our paper's setting, since we are focusing on systems whose behavior at $s\to\infty$ is potentially nonstandard, as encoded by a nonzero relative degree.

To circumvent this problem, we propose a practical strategy relying on \emph{piecewise} approximation, where the surrogate model is evaluated by using one of two formulas: one that is accurate for low frequencies, namely, the barycentric form, \eqref{eq:bary}, and another that is accurate for large frequencies, which we will call ``asymptotic form''.

The following result provides the foundations for such an asymptotic form for any rational function in barycentric form.

\begin{lemma}\label{lem:asympt}
As in Theorem~\ref{th:degree}, let the non-trivial rational function $r$ in \eqref{eq:bary} have exact type $(m-\mu,m-\nu)$. As $s\to\infty$, we have the asymptotic expansion
\begin{equation}\label{eq:leading}
	r(s)=\frac{\sum_{l=0}^\infty\left(\sum_{k=0}^mw_kf_ks_k^{\mu+l}\right)s^{-l}}{\sum_{l=0}^\infty\left(\sum_{k=0}^mw_ks_k^{\nu+l}\right)s^{-l}}s^{\rdeg(r)}.
\end{equation}
\end{lemma}

\begin{proof}
By Lemma~\ref{lem:eqasympt}, we have the expansions
\begin{equation*}
	\sum_{k=0}^m\frac{w_kf_k}{s - s_k}=\sum_{l=0}^\infty\left(\sum_{k=0}^mw_kf_ks_k^l\right)\frac1{s^{l+1}}\quad\text{and}\quad\sum_{k=0}^m\frac{w_k}{s - s_k}=\sum_{l=0}^\infty\left(\sum_{k=0}^mw_ks_k^l\right)\frac1{s^{l+1}}.
\end{equation*}
By assumption, the first $\mu$ (respectively, $\nu$) terms vanish from the sum in the first (respectively, second) identity above. Recalling that $\rdeg(r)=\nu-\mu$, the claim follows by re-indexing the above series and plugging them into the barycentric form \eqref{eq:bary}.
\end{proof}

The desired asymptotic form can then be obtained by simply truncating the two series appearing in \eqref{eq:leading}, both having running index $l$, at some (user-defined) nonnegative order $N$.

The last item needed for our piecewise rational formula is the cutoff radius $R_{\text{cutoff}}$, such that the barycentric, resp.~asymptotic, form is used to evaluate $r(s)$ whenever $|s|\leq R_{\text{cutoff}}$, resp.~$|s|>R_{\text{cutoff}}$. While the value of $R_{\text{cutoff}}$ may ultimately be selected by a user of our method, we outline here a heuristic way of choosing it in practice.

For this, we compare the extrapolation error incurred with the barycentric form (including effects due to numerical cancellation) with the error that the asymptotic form incurs due to truncation of the series in Lemma~\ref{lem:asympt}. We choose $R_{\text{cutoff}}$ as the frequency $s$ where the two above-mentioned errors are equal. Of course, using the exact errors is unfeasible, as we would need to know the exact, noise-free frequency response at arbitrary frequencies. Instead, we rely on some heuristics.

First, let $T=\max_j|s_j'|$ be the largest frequency for which data is available. When using the barycentric form at large frequencies ($|s|>T$), we assume that the extrapolation error increases following the power law
\begin{equation}\label{eq:trunc}
\text{err}_{\text{bary}}(s)\simeq\epsilon\left(\frac{|s|}{T}\right)^{|\rdeg(r)|}.
\end{equation}
Above, $\epsilon$ is an upper bound for the relative approximation error over the frequency region where data is available. For instance, it may be cheaply computed as the largest approximation error over the training dataset. The exponent ``$|\rdeg(r)|$'' is consistent with Lemma~\ref{lem:asympt}, as well as with our numerical findings in the next section.

On the other hand, by direct inspection of the numerator and denominator of \eqref{eq:leading}, we can verify that the truncation error (at order $l=N$) incurred by the asymptotic form is
\begin{equation*}
\text{err}_{\text{asympt}}(s)\simeq\left(\frac{T}{|s|}\right)^{N+1}.
\end{equation*}
For the sake of simplicity, we are assuming the scaling constant in front of the $|s|^{-N-1}$ scaling to be simply $T^{N+1}$. Again, our numerical tests provide evidence supporting such a choice. It is now easy to verify that $\text{err}_{\text{bary}}$ and $\text{err}_{\text{asympt}}$ are equal at
\begin{equation}\label{eq:cutoff}
|s|=R_{\text{cutoff}}:=T\epsilon^{-1/(|\rdeg(r)|+N+1)}.
\end{equation}

\section{Data-driven identification of the relative degree}\label{sec:DegreeAutomatic}

In the previous section, we described how to build a rational approximation with a given degree from available data. However, in the task of \emph{data-driven system identification}, it is often impossible to know the ``correct'' relative degree in advance: in many settings, the relative degree of the function that generated the frequency-domain data (i.e., the \emph{relative degree} of the underlying dynamical system, cf.~Remark~\ref{rem:degreesystem}) is unknown. As such, we face the task of identifying such a degree from the frequency-domain data.

For this, we rely on \emph{model selection}: first, we approximate the data with several rational functions $\{r^{(\delta)}\}_\delta$ of different relative degrees $\delta$, computed with the strategy outlined in Section~\ref{sec:RatApproxConstrained}. Among those, we select the ``best'' approximation, in a sense to be specified. The relative degree of the selected rational approximation represents our guess for the ``exact'' relative degree of the underlying system. The only two ingredients that remain to be specified are: (i) how to select the list of potential relative degrees, and (ii) how to choose the ``best'' rational fit of the data.

A crucial ingredient for answering both questions is a metric to assess the quality of a rational function. Specifically, for our discussion, it is enough to have a criterion to determine which of \emph{two} rational approximations is better. While this will ultimately depend on the specific algorithm used to build the rational functions, cf.~Section~\ref{sec:Extensions}, we describe here a strategy that can be applied in combination with AAA.

Let $r^{(\delta)}$ and $r^{(\delta')}$ be arbitrary rational functions with relative degrees $\delta$ and $\delta'$, with $(m^{(\delta)}+1)$ and $(m^{(\delta')}+1)$ barycentric coefficients (cf.~\eqref{eq:bary}), respectively. We define their corresponding $\ell^\infty$ approximation errors $e_\infty^{(\delta)}$ and $e_\infty^{(\delta')}$ as the largest (relative) approximation errors achieved by the two functions over the training dataset.

Assuming both $r^{(\delta)}$ and $r^{(\delta')}$ are computed using AAA, so that $\max\{e_\infty^{(\delta)},e_\infty^{(\delta')}\}<\varepsilon_{\text{AAA}}$, we say that
\begin{equation}\label{eq:criterion_m}
r^{(\delta)}\text{ is better than }r^{(\delta')}\text{ iff }\begin{cases}
	m^{(\delta)}<m^{(\delta')}\quad&\text{if }m^{(\delta)}\neq m^{(\delta')},\\
	|\delta|>|\delta'|\quad&\text{if }m^{(\delta)}=m^{(\delta')}\text{ and }|\delta|\neq|\delta'|,\\
	e_\infty^{(\delta)}<e_\infty^{(\delta')}\quad&\text{if }m^{(\delta)}=m^{(\delta')}\text{ and }|\delta|=|\delta'|.
\end{cases}
\end{equation}
This can be motivated as follows:
\begin{itemize}
\item \rev{Since both candidates satisfy the AAA tolerance $\varepsilon_{\text{AAA}}$, their $\ell^\infty$ errors are typically indistinguishable. We thus compare their ``complexities'' $m$, preferring the function that attains the tolerance with fewer support points.}
\item \rev{If two candidates have the same $m$, we prefer the one with larger relative degree $|\delta|$. By Theorem~\ref{th:degree}, this corresponds to fewer effective degrees of freedom in $r^{(\delta)}$ (since a degree $\delta$ corresponds to $|\delta|$ constraints on the barycentric coefficients), and thus a simpler model in our setting.}
\item If both complexity and degree are tied, we fall back to comparing the two rational functions through their approximation error, \rev{no other metric being available for a comparison}.
\end{itemize}
\begin{remark}\label{rem:criterion}
\rev{Criterion~\eqref{eq:criterion_m} is deliberately simple: it is lightweight and easy to interpret, and it proves remarkably effective in our numerical tests. At the same time, it is only a heuristic, whose correctness can be shown only in limited cases, cf.~Lemma~\ref{lem:criterion} below. In particular, when data is noisy or when small perturbations cause AAA to terminate with different values of $m$, the criterion may be inconclusive. In such cases, supplementary checks, e.g., via cross-validation, may help identify the correct degree in a more robust way. Developing robust alternatives with solid theoretical foundations is a promising direction for future work.}
\end{remark}

We can now describe our model selection algorithm. We begin with a relative degree of $\delta=0$ and then progressively increment it. As soon as $r^{(\delta)}$ is not better (in the sense of \eqref{eq:criterion_m}) than $r^{(\delta-1)}$, we stop our search, selecting $r^{(\delta-1)}$ temporarily as the best approximation. This process is then repeated with negative relative degrees $\delta=-1,-2,\ldots$, comparing $r^{(\delta)}$ with $r^{(\delta+1)}$ and stopping as soon as the best approximation $r^{(\delta+1)}$ with non-positive relative degree is found. Finally, the best approximations with non-negative and non-positive degrees are compared, and the better one is chosen as the ultimate ``best''. Some sample pseudo-code for this strategy is provided in Algorithm~\ref{algo}.

\begin{algorithm}
\caption{Data-driven AAA-based degree-identification}
\label{algo}
\begin{algorithmic}
	\REQUIRE{data points $D:=\{(s_j',f(s_j'))\}_j$ and AAA tolerance $\varepsilon_{\text{AAA}}>0$}
	\STATE{Build baseline $r^{(0)}:=\texttt{AAA}(D,\varepsilon_{\text{AAA}},0)$}
	\FOR{$\delta=1,2,...$}
	\STATE{Build $r^{(\delta)}:=\texttt{AAA}(D,\varepsilon_{\text{AAA}},\delta)$ as in Section~\ref{sec:RatApproxConstrained}}
	\IF{$r^{(\delta-1)}$ is better than $r^{(\delta)}$, cf.~\eqref{eq:criterion_m}}
	\STATE{Set $\delta^\star:=\delta-1$ and \textbf{break} out of the for-loop}
	\ENDIF
	\ENDFOR
	\FOR{$\delta=-1,-2,...$}
	\STATE{Build $r^{(\delta)}:=\texttt{AAA}(D,\varepsilon_{\text{AAA}},\delta)$ as in Section~\ref{sec:RatApproxConstrained}}
	\IF{$r^{(\delta+1)}$ is better than $r^{(\delta)}$, cf.~\eqref{eq:criterion_m}}
	\STATE{Set $\delta_\star:=\delta+1$ and \textbf{break} out of the for-loop}
	\ENDIF
	\ENDFOR
	\IF{$r^{(\delta^\star)}$ is better than $r^{(\delta_\star)}$, cf.~\eqref{eq:criterion_m}}
	\RETURN{$\delta^\star$}
	\ELSE
	\RETURN{$\delta_\star$}
	\ENDIF
\end{algorithmic}
\end{algorithm}

\subsection{Complexity}
\rev{Each call to the AAA algorithm on a dataset of size $m'$ with $m$ support points has asymptotic cost $m'm^3$ \cite{NST18}. Since Algorithm~\ref{algo} tests all candidate degrees in the range $[\delta_\star-1,\delta^\star+1]$, the total cost is $(\delta^\star-\delta_\star+3)m'm^3$. As such, when the true relative degree is small, Algorithm~\ref{algo} costs only a modest multiple of a single (constrained or unconstrained) AAA call. For higher-degree systems, however, the extra factor may become significant.}

\rev{Several optimizations can be carried out to reduce the overall cost. For instance, one may easily run the ``$\delta>0$'' cases in parallel to the ``$\delta<0$'' ones. Moreover, computations can often be reused across consecutive degrees: by Remark~\ref{rem:deg}, the first $|\delta|$ iterations of $(D,\varepsilon_{\text{AAA}},\delta)$ are identical to those of \\ $\texttt{AAA}(D,\varepsilon_{\text{AAA}},\textup{sign}(\delta)(|\delta|+1))$, so the corresponding intermediate support points and SVDs need not be recomputed.}

\rev{Finally, we also note that, in many applications, the cost of Algorithm~\ref{algo} is negligible compared with the expense of gathering or producing the data that is given as input to the method. For instance, this is true when data is generated by high-resolution numerical simulations or measured through experiments. In such settings, the additional factor introduced by degree identification (with respect to a standard rational approximation) is unlikely to be a practical bottleneck. Empirical evidence for this claim is provided in Section~\ref{subsec:SOTA}.}

\subsection{Consistency}
\rev{Before proceeding, we present this result as a partial justification of \eqref{eq:criterion_m}.}
\begin{lemma}\label{lem:criterion}
\rev{Assume that the input data $\{(s_j',f(s_j'))\}_{j=1}^{m'}$ is generated by a function $f$ that is rational with exact type $(a,b)$, with $m'>\max\{a,b\}+1$. In exact arithmetic, setting $\varepsilon_{\text{AAA}}=0$ makes Algorithm~\ref{algo} yield the exact relative degree $(a-b)$.}
\end{lemma}
\begin{proof}
\rev{Assume without loss of generality that $a\geq b$. Consider first the trial degrees $0\leq\delta\leq a-b$, which have the correct sign and do not overestimate the exact relative degree.
	
	Since $\varepsilon_{\text{AAA}}=0$ and the distinct sample points $\{s_j'\}_j$ are more than $(a+1)$, any call to $\texttt{AAA}(D,\varepsilon_{\text{AAA}},\delta)$ yields an exact rational approximation $r^{(\delta)}=f$ of (not necessarily exact) rational type $(a,a-\delta)$. All such rational functions require $m^{(\delta)}+1=a+1$ barycentric coefficients to attain the tolerance, since (i) no rational function of numerator degree $<a$ can match $f$ at $(a+1)$ distinct points and (ii) AAA achieves exactness with the minimal number of support points thanks to the interpolation property. By the second condition in \eqref{eq:criterion_m}, Algorithm~\ref{algo} selects $r^{(a-b)}$ as ``best'' among such $(a-b+1)$ rational functions, since it has the largest degree.}

\rev{At the next iteration of Algorithm~\ref{algo}, $\delta=a-b+1$ and $\texttt{AAA}(D,\varepsilon_{\text{AAA}},a-b+1)$ also yields an exact rational approximation $r^{(a-b+1)}=f$. However, $m^{(a-b+1)}+1=a+2$ barycentric coefficients are now required to produce a rational function of sufficient (inexact) type $(a+1,b)$. As such, $r^{(a-b)}$ remains the current best by the first condition in \eqref{eq:criterion_m}, since $m^{(a-b)}<m^{(a-b+1)}$.}

\rev{A similar reasoning shows that, when trying the first negative degree $\delta=-1$, $\texttt{AAA}(D,\varepsilon_{\text{AAA}},-1)$ yields a exact approximation $r^{(-1)}=f$ of (inexact) type $(a,a+1)$, requiring one extra barycentric coefficient because of the denominator: $m^{(-1)}+1=a+2$. As such, $r^{(0)}$ is the best rational approximation with non-positive degree.}

\rev{If $a=b$, then the best non-negative-degree and non-positive-degree candidates coincide, making the claim trivially true. If $a>b$, the comparison of the current candidates $r^{(a-b)}$ and $r^{(0)}$ yields $r^{(a-b)}$ as the best rational approximation according to the second condition in \eqref{eq:criterion_m}, since $a-b>0$.}
\end{proof}
\rev{It is clear that the above result has limited relevance in realistic scenarios, since a null AAA tolerance is extremely impractical. However, Lemma~\ref{lem:criterion} shows that \eqref{eq:criterion_m} is, in some sense, \emph{mathematically consistent}, at least when data is not affected by noise.}

\section{Extensions to least-squares rational approximation}
\label{sec:Extensions}

Although we have focused on AAA throughout the paper, most of our discussion generalizes to any rational approximation strategy that employs the barycentric rational format. Some notable mentions are the vector fitting (VF) algorithm \cite{VF}, the Antoulas-Anderson (AA) approach \cite{AA86}, and the minimal rational interpolation (MRI) method \cite{P20}. As already mentioned in Section~\ref{sec:RatApprox}, VF stands out from the rest because it is not interpolatory, since its support points are chosen not as a subset of the sample points, but through an iterative procedure that aims at progressively improving their location. This lack of interpolation makes VF more robust whenever the data is affected by noise. Indeed, in such cases, the support values $f_k$ used in rational interpolation are affected by noise, generally leading to a lower overall accuracy. For this reason, it is useful to study whether a modification of our data-driven degree-identification approach in Algorithm~\ref{algo} can be developed, with VF (or other non-interpolatory rational approximation methods) replacing AAA.

We note that our theoretical results from Section~\ref{sec:RatDegree} trivially generalize to rational functions in non-interpolatory barycentric form, as we proceed to state rigorously.
\begin{corollary}\label{cor:Extensions}
Theorem~\ref{th:degree} and Lemma~\ref{lem:asympt} remain valid for functions $r$ in barycentric form \eqref{eq:baryfull}, as long as $w_kf_k$ and $w_k$ are replaced by $n_k$ and $d_k$, respectively.

The same is true for Corollary~\ref{th:vandermonde}, although the claim concerning the numerator degree must be stated without involving the matrix $F$, as ``If the vector $\mathbf{n}=[n_0,\ldots,n_m]^\top$ is nonzero and lies in the null space of $V_{\mu-1}^\top$, then the exact degree of the numerator of $r$ in \eqref{eq:bary} is at most $m-\mu$.''.
\end{corollary}
\begin{proof}
The claim follows trivially by inspection.
\end{proof}

Moreover, we recall that, as discussed in Section~\ref{sec:RatApprox}, all the above-mentioned methods rely on a least-squares problem similar to \eqref{eq:LSAA}, namely, \eqref{eq:LSVF}, to characterize the approximation's barycentric coefficients. As such, enforcing a target degree for the rational approximant is just as simple as done in Section~\ref{sec:RatApproxConstrained} for AAA, by restricting the barycentric coefficients to lie in the null space of a suitable (generalized) Vandermonde matrix.

The only part of our discussion that \emph{cannot} be easily generalized to other rational approximation methods is the automatic degree identification presented in Section~\ref{sec:DegreeAutomatic}. More specifically, the main difficulty is generalizing criterion \eqref{eq:criterion_m}, which is needed to assess the quality of rational functions and, ultimately, for degree identification. Indeed, \eqref{eq:criterion_m} uses the complexity of a rational function, as encoded by the number $m$ of its barycentric coefficients, as the main indicator of ``badness'' of a rational function, neglecting the rational function's approximation error. This is done by assuming that complexity provides much more information on the quality of a rational function than the approximation error. This is true for AAA, since the method is designed to adaptively explore different complexities. However, most other rational approximation methods (including the above-mentioned ones) \emph{fix} the complexity in advance, thus making it an improper indicator of quality.

One might think that, since complexity is fixed, it might be possible to fall back on the approximation error as a way to compare rational functions. However, this is not the case, as we proceed to explain. Recall that the barycentric coefficients are found by minimizing the approximation error\footnote{Admittedly, the rational approximant minimizes the $\ell^2$ error through \eqref{eq:LSAA}, which, in theory, would allow us to use other kinds of error (e.g., the $\ell^\infty$ one) to compare rational functions without incurring the mentioned error hierarchy. Still, in our experience, imposing degree constraints usually increases \emph{both} $\ell^2$ and $\ell^\infty$ approximation errors, so that the outlined difficulties are still present.}, cf.~\eqref{eq:LSAA}. Any additional degree constraint reduces the size of the ``feasible set'' where the barycentric coefficients are sought, thus making the minimal value of the target function (i.e., the approximation error) increase. As such, using the approximation error to perform comparisons in Algorithm~\ref{algo} makes the method always select the default degree-$0$ rational approximation, it being the one that achieves the smallest error.

As such, it may appear impossible to compare rational functions of different degrees built through, e.g., VF. However, this can be achieved by leveraging a form of model selection. To this aim, we note that VF has already been coupled with model selection strategies \cite{VFMS,LFMS}, as a way to overcome one of its main limitations, namely, the need to fix the ``surrogate complexity'' $m$ in advance. Such model selection strategies generally work by using VF to independently build several rational approximations with different values of $m$, and then selecting the rational approximant characterized by the lowest approximation error.

\rev{As a prototypical example of such a model-selection technique, consider a strategy where, given a target error tolerance $\varepsilon_{\text{VF}}>0$ and a target relative degree $\delta$, one sequentially applies VF with rational surrogates of increasing complexity $m=1,2,\ldots$, stopping as soon as the uniform approximation error $\ell^\infty$ is below $\varepsilon_{\text{VF}}$. In a very practical sense, this strategy mimics AAA, cf.~Algorithm~\ref{algo:aaa}, with two main differences: (i) the selection of the support points is performed through the SK iterations rather than by selecting the sample point with the largest error \cite{VF} and (ii) the (in this case) detrimental interpolation property is weakened to a more robust least-squares approximation, cf.~\eqref{eq:LSVF}.}

\rev{Using the above strategy, one \emph{can} effectively compare rational functions based on their approximation error, since the error is not monotonic with respect to the complexity. This is in contrast with what we mention for our degree-identification routine, where the approximation error behaves monotonically with respect to (the absolute value of) the degree.} This means that, given some data, it \emph{is} possible to perform a form of VF where different complexities are adaptively explored, in a flavor similar to AAA, enabling the use of criterion \eqref{eq:criterion_m} to compare rational functions. Such an ``adaptive-complexity'' flavor of VF can then be used to replace AAA in our degree-identification strategy, namely, Algorithm~\ref{algo}. Our numerical tests, presented in the next section, show that this approach is promising when dealing with noisy data.

\rev{In practice, we envision a unified strategy where AAA serves as a fast and effective default method for clean or synthetic data, while the above-described, adaptive-complexity variant of VF acts as a more robust counterpart. When the input data is affected by an unclear amount of noise, one may either run the two approaches in parallel or use VF selectively as a fallback when AAA produces unclear results, e.g., if Algorithm~\ref{algo} returns the ``default'' value $\delta=0$. In both cases, the degree-identification criterion \eqref{eq:criterion_m} remains applicable. This being said, this method is not immune to failure: as we discuss in Section~\ref{subsec:noisy}, unreliable identification may occur with highly noisy data or with data restricted to an excessively narrow frequency window.}

\section{Numerical examples}
\label{sec:NumericalExamples}

In order to assess the effectiveness of our method, we test it on a plethora of dynamical systems. Specifically, we consider:
\begin{itemize}
\item A simple mechanical system that models the elastic interaction of point masses, as illustrated in Figure~\ref{fig_2Car_example}. More details on the model are given in Appendix~\ref{ap:cars}. Here it is enough to mention that, given a number $n\geq 2$ of masses, the model comes in two forms: a ``forward'' model with negative degree $-2n$ and an ``inverted'' model with positive degree $2n$.
\item \rev{Two} Oseen flow problem\rev{s} subject to Dirichlet-type boundary control. The map from the control to the fluid pressure exhibits the characteristics of a DAE of index 2 (degree 1). See, e.g., \cite{GosH23} for the linearized case and \cite{AltH19} for the relevant considerations of the nonlinear Navier-Stokes equations. In our numerical examples, we consider the linear setup described in \cite{GosH23}, which models how modulations in the inflow conditions affect pressure measurements in the (linearized) flow past a cylinder. \rev{We look at two \emph{Reynolds numbers} $\mathsf{Re}=20$ (which was already considered in \cite{GosH23}) and $\mathsf{Re}=30$ (which requires a finer mesh than the former case). These examples are particularly challenging because their state spaces are of sizes $\sim 5\cdot10^4$ and $\sim 6\cdot10^4$, respectively.} This makes the acquisition of data rather computationally expensive. \rev{Moreover, stronger convection effects, as encoded in larger Reynolds numbers, may correspond to an increased complexity in the transfer function due to oscillations and a generally less stable behavior of the system.}
\item Many examples from the SLICOT benchmark collection \cite{morChaV02}, whose relative degrees are either -1 or 1.
\end{itemize}
Code and data that allow reproducing the numerical tests are available at~\cite{CodZen24}.

As a synthetic way to evaluate our method, we check if it can correctly identify the degree of each of the above-mentioned systems. In all cases, our training samples are located in a very limited frequency band, which is selected case by case, ensuring that the natural frequency of each system is included. In particular, to ensure the fairness of our testing, we tried to select sampling frequencies as low as possible. In this, we are mimicking practical situations where gathering samples at high frequencies is unfeasible, e.g., due to numerical instabilities (for simulated data) or to limitations of the experimental apparatus (for measured data). Obviously, sampling at larger frequencies would have made the task of degree identification simpler, since more information on the frequency-response behavior at $s\to\infty$ would be available.

\begin{table}[htbp]
\caption{Summary of degree-identification results. Ticks and crosses mark correct and incorrect identifications of the exact relative degrees, respectively. Incorrectly identified degrees are reported in parentheses.}
\label{tab:results}
\begin{center}
	\begin{tabular}{c|c|c|c|c|c|c|}
		\multicolumn{2}{c|}{} & Exact & Sampled & \multicolumn{3}{c|}{Relative degree estimate}\\
		\cline{5-7}
		\multicolumn{2}{c|}{Test case} & relative & frequency & \multicolumn{2}{c|}{with AAA} & with VF\\
		\cline{5-7}
		\multicolumn{2}{c|}{} & degree & range & noiseless & \multicolumn{2}{c|}{noisy} \\
		\hline
		\multicolumn{2}{c|}{Forward 2-mass} & -4 & {\small $10^{-2}\div10^{0}$} & \cmark & \xmark (0) & \cmark\\
		\hline
		\multicolumn{2}{c|}{Forward 3-mass} & -6 & {\small $10^{-2}\div10^{0}$} & \cmark & \xmark (-1) & \xmark (-1)\\
		\hline
		\multicolumn{2}{c|}{Inverted 2-mass} & 4 & {\small $10^{-2}\div10^{0}$} & \cmark & \cmark & \cmark\\
		\hline
		\multicolumn{2}{c|}{Inverted 3-mass} & 6 & {\small $10^{-2}\div10^{0}$} & \xmark (0) & \xmark (2) & \xmark (1)\\
		\hline
		\multicolumn{2}{c|}{Oseen flow\rev{, $\mathsf{Re}=20$}} & 1 & {\small $\rev{10^{-1}}\div10^{1}$} & \cmark & \xmark (\rev{-1}) & \rev{\xmark (0)}\\
		\hline
		\multicolumn{2}{c|}{\rev{Oseen flow, $\mathsf{Re}=30$}} & \rev{1} & \rev{\small $10^{-1}\div10^{1}$} & \rev{\cmark} & \rev{\xmark (0)} & \rev{\cmark}\\
		\hline
		\multirow{9}{*}{\rotatebox{90}{SLICOT}} & \texttt{eady} & -1 & {\small $10^{0}\div10^{3}$} & \cmark & \xmark (-2) & \cmark\\
		\cline{2-7}
		& \texttt{peec} & -1 & {\small $10^{3}\div10^{4}$} & \cmark & \xmark (0) & \xmark (-3)\\
		\cline{2-7}
		& \texttt{fom} & -1 & {\small $10^{1}\div10^{2}$} & \cmark & \xmark (0) & \cmark\\
		\cline{2-7}
		& \texttt{pde} & -1 & {\small $10^{1}\div10^{4}$} & \cmark & \xmark (0) & \cmark\\
		\cline{2-7}
		& \texttt{heat-disc} & -1 & {\small $10^{0}\div10^{2}$} & \cmark & \cmark & \cmark\\
		\cline{2-7}
		& \texttt{beam} & -1 & {\small $10^{-2}\div10^{2}$} & \cmark & \cmark & \cmark\\
		\cline{2-7}
		& \texttt{mna\_1}, entry $(1,1)$ & 1 & {\small $10^{11}\div10^{12}$} & \cmark & \cmark & \cmark\\
		\cline{2-7}
		& \texttt{mna\_1}, entry $(1,3)$ & -1 & {\small $10^{11}\div10^{12}$} & \cmark & \xmark (0) & \cmark\\
		\cline{2-7}
		& \texttt{mna\_1}, entry $(2,3)$ & 1 & {\small $10^{11}\div10^{12}$} & \xmark (-2) & \xmark (0) & \xmark (-1)\\
		\hline
	\end{tabular}
\end{center}
\end{table}

The results are shown in Table~\ref{tab:results} and are discussed in detail in the upcoming sections. To assess numerical robustness, we use both noiseless and noisy data. The noise is manufactured in a multiplicative way: the $j$-th data value is perturbed as
\begin{equation*}
f_{\textup{noisy}}(s_j')=f(s_j')\Big(1+10^{-6}\xi_j\Big),
\end{equation*}
with $\xi_1,\xi_2,\ldots$ being independent samples from a uniform random variable taking values in $[-1,1]$.

\subsection{Results with AAA}

First, we apply our method exactly as presented in Algorithm~\ref{algo}, relying on AAA for rational approximation. We set the AAA tolerance to $10^{-6}$ and $10^{-4}$ in the noiseless and noisy cases, respectively. The results are reported in the third-to-last and second-to-last columns of Table~\ref{tab:results}. We see that our method is generally reliable in the noiseless case, with mostly correct degree identifications across the board. The two failures are discussed and explained in detail below.

When the data is noisy, the algorithm fails in more than two-thirds of the tests. This is due to obvious limitations in applying AAA, an interpolatory algorithm, to noisy data. In this sense, we look at the noisy tests' results as a way to verify that our method \emph{can} work in unfavorable environments, rather than breaking down completely. Indeed, in almost all cases of degree mislabeling due to noise, the degree is underestimated as opposed to overestimated (in magnitude). In fact, our algorithm estimates the degree to be its default zero value in most misidentified cases.

\subsubsection{Forward 3-mass system}
We now proceed to discuss some selected results in further detail, starting from the noiseless forward 3-mass system. The quality of the approximation can be observed in Figure~\ref{fig:numresults-fcars} (left). Our adaptive degree-identification approach is clearly superior to ``standard'' AAA, with the latter's approximate frequency response displaying a characteristic saturation at large frequencies, due to its ``default'' zero relative degree. Note, in particular, how our approach detects the correct relative degree from data on a frequency range where, seemingly, no qualitative indication of the correct degree-$(-6)$ scaling is present: the training data follows much more closely an inverse-quadratic scaling, cf.~the reference curve in the top-left plot.

\begin{figure}[htbp]
\centering
\includegraphics[scale=.85]{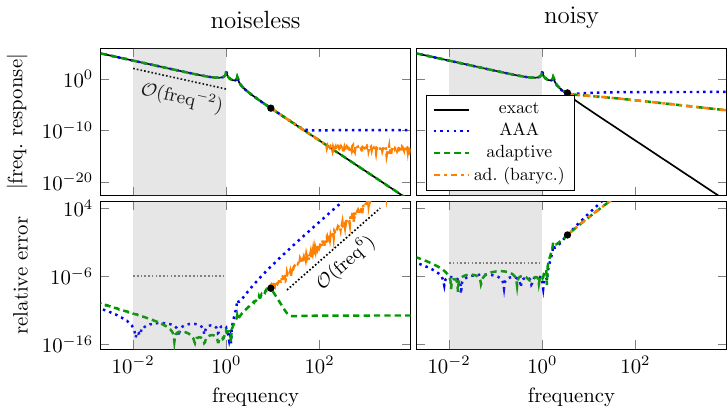}
\caption{Numerical results for the forward transfer function of the $3$-mass chain. The gray band is the sampled frequency range. Black dots denote the point $R_{\text{cutoff}}$, where the barycentric rational form gives way to the asymptotic rational form. The orange dash-dotted curves pertain to the adaptive method's results using purely the barycentric form.}
\label{fig:numresults-fcars}
\end{figure}

It is interesting to note that, although our approach builds a rational function with the exact relative degree $-6$ (automatically identified by our algorithm), numerical cancellation in the barycentric form leads to a saturation effect similar to AAA's. This is evidenced by a drastically increasing approximation error for frequencies larger than $10^{1}$, in agreement with \eqref{eq:trunc}. However, the asymptotic form from Lemma~\ref{lem:asympt}, which we use (with series truncated at $N+1=11$ terms) for frequencies larger than $R_{\text{cutoff}}$ (computed as in \eqref{eq:cutoff}), manages to recover a stable evaluation even at large frequencies. Still concerning the use of the asymptotic form, we note how the cutoff radius $R_{\text{cutoff}}$ is chosen well, ensuring that, at $|s|=R_{\text{cutoff}}$, cancellation errors in the barycentric form are balanced with truncation errors in the asymptotic form.

The results for the noisy case, in Figure~\ref{fig:numresults-fcars} (right), are worse, since neither approach is able to correctly identify the correct relative degree of the system. Ultimately, this is due to interpolating the noise within the sampled frequency range, preventing good extrapolation properties. Specifically for our approach, a correct automatic identification of the degree is hindered by the fact that the noise is spatially uncorrelated, hence difficult to fit with rational functions.

\subsubsection{Inverted 3-mass system}\label{subsec:inv}

We now move to our first failed test, involving the inverted 3-mass system. Since our automatic degree identification returns the default zero degree, the corresponding rational approximation coincides with that obtained by AAA. As shown in Figure~\ref{fig:numresults-icars} (left), the approximation quality quickly degrades as the frequency increases, as a consequence of the badly identified degree.

\begin{figure}[htbp]
\centering
\includegraphics[scale=.85]{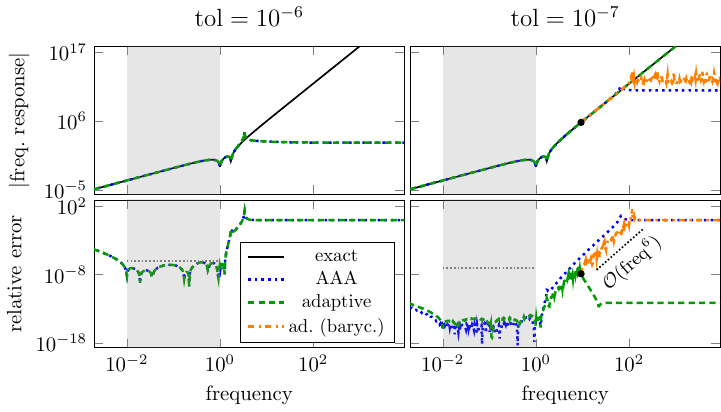}
\caption{Numerical results for the inverted transfer function of the $3$-mass chain. The gray band is the sampled frequency range. Black dots denote the point $R_{\text{cutoff}}$, where the barycentric rational form gives way to the asymptotic rational form. The orange dash-dotted curves pertain to the adaptive method's results using purely the barycentric form.}
\label{fig:numresults-icars}
\end{figure}

A closer inspection allowed us to identify why the degree was misidentified. Running ``standard'' AAA (relative degree $\delta=0$) with a tolerance of $10^{-6}$ leads to the selection of 6 support points, with a rational approximant of type $(5,5)$. However, this is barely enough to attain the tolerance, with the largest approximation error over the sample points being $6.3\cdot10^{-7}$. On the other hand, building a rational approximation of type $(5,4)$ (relative degree $\delta=-1$) attains an error above the tolerance, which makes our algorithm reject positive degrees.

To recover a correct behavior, it was enough to slightly decrease the AAA tolerance to $10^{-7}$. We display the corresponding results in Figure~\ref{fig:numresults-icars} (right). This empirically suggests that, with the aim of further increasing the robustness of our approach, it might prove useful to use different AAA tolerances at different stages of the algorithm: a base tolerance for the call to ``standard'' AAA (relative degree $\delta=0$) and a slightly larger tolerance for all other calls to AAA with nonzero degree constraints. Note, however, that this modification could make our method more prone to overestimating the degree (in absolute terms).

\subsubsection{SLICOT MNA system}\label{subsec:mna}

Finally, we look at our other failed test, involving a modified-nodal-analysis system from the SLICOT library \cite{morChaV02}. First, as indicated in Table~\ref{tab:results}, the relative degrees of the $(1,1)$ and $(1,3)$ entries of the transfer-function matrix are correctly identified. This is also the case for most other entries of the transfer function, which we do not include in our showcase for conciseness.

However, as we see in Figure~\ref{fig:numresults-mna1}, this has no implications on the accuracy of the rational approximation outside the sampled frequency range\footnote{The error is sometimes above the prescribed tolerance even within the sampling range, as a consequence of the fact that we are using a finer frequency grid for making these plots than for training the surrogate.}. We can observe that large errors are incurred in the secondary peaks of the transfer function. Notably, the asymptotic scaling for large frequencies is well identified, despite errors in the estimation of the \emph{constant in front of the scaling}. These results are quite remarkable, since they show how our method is able to decouple \emph{degree identification} from \emph{transfer-function approximation at large frequencies}.

\begin{figure}[htbp]
\centering
\includegraphics[scale=.85]{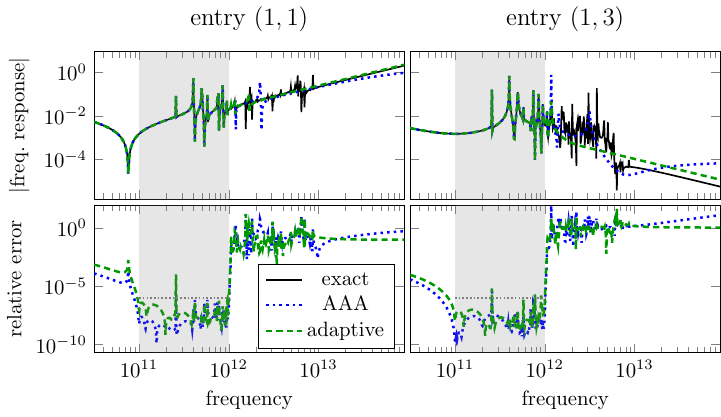}
\caption{Numerical results for two entries of the \texttt{MNA\_1} transfer-function matrix. The gray band is the sampled frequency range.}
\label{fig:numresults-mna1}
\end{figure}

However, there are limitations, as evidenced by the misidentification of the degree of the $(2,3)$ entry of the transfer-function matrix. The corresponding (lack of) approximation accuracy can be observed in Figure~\ref{fig:numresults-mna2} (left). This is obviously an example of the above-mentioned decoupling between degree identification and high-frequency transfer-function approximation going wrong. Otherwise stated, the information on a very limited frequency range, which misses some dominant peaks of the transfer function, was not enough to draw correct conclusions on the asymptotic scaling of the transfer function. The simplest way to fix this, recovering a correct degree identification, is to increase the largest sampled frequency. We show in Figure~\ref{fig:numresults-mna2} (right) the results for a slightly larger sampling window (the largest sampled frequency is $10^{12.25}$ instead of $10^{12}$), which, although not so qualitatively different from the results plotted in Figure~\ref{fig:numresults-mna2} (left), display a correct degree identification.

\begin{figure}[htbp]
\centering
\includegraphics[scale=.85]{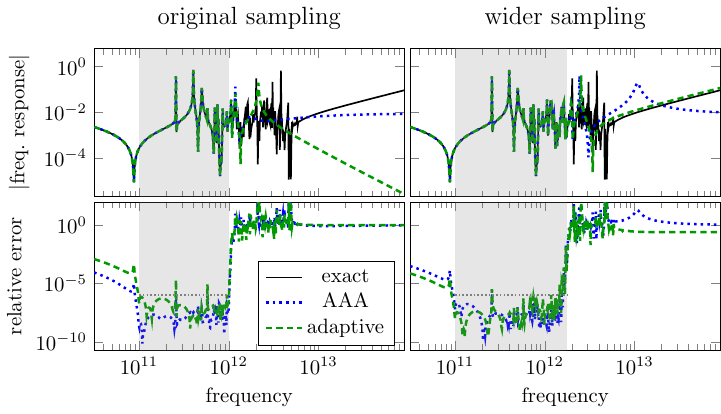}
\caption{Numerical results for the entry $(2,3)$ of the \texttt{MNA\_1} transfer-function matrix. The gray band is the sampled frequency range.}
\label{fig:numresults-mna2}
\end{figure}

\subsection{Results with VF}\label{subsec:noisy}

In Section~\ref{sec:Extensions}, we discussed how to apply our degree-identification algorithm with least-squares rational approximation, as opposed to interpolation. Non-interpolatory least-squares rational approximation is definitely a more valid choice in the noisy setting. As a way to assess this empirically, we repeat our numerical experiments, replacing AAA with a simplified version of VF.

To build a rational approximation, we fit the data (in a least-squares sense) using VF, although, for the sake of simplicity and to streamline the results, we run a single SK iteration. Given data $\{(s_j',f(s_j'))\}_{j=1}^{m'}$ from the frequency range $t=|s_1'|\leq\ldots\leq|s_{m'}'|=T$, we choose geometrically spaced support points \\ $\{0.9s_1'(1.2T/t)^{k/m}\}_{k=0}^m$. (The scaling factors $0.9$ and $1.2$ have been added to ensure that the support points are disjoint from the sample points.) As discussed in Section~\ref{sec:Extensions}, we perform model selection by progressively increasing the number of support points until the approximation error is uniformly below a $10^{-4}$ tolerance.

The results are reported in the last column of Table~\ref{tab:results}. We can observe a much higher success rate with respect to AAA, with correct identification of the relative degree in two-thirds of the test cases. Although we do not present detailed results here, we note that there is numerical evidence to suggest that the misidentifications of the degree happen mainly because of sampling regions that are too narrow, similarly to what was presented in Section~\ref{subsec:mna} concerning AAA. Specifically, we have found that enlarging the sampling regions leads to an increased prediction accuracy, as is reasonable to expect.

Ultimately, since our degree-identification method is effectively a form of extrapolation, it is impossible to achieve complete reliability in the case of noisy data. Indeed, as our results show, noise might lead the identification procedure astray, especially if the data comes from a very narrow and low-frequency window. At the same time, we note that improved robustness may be achieved through some form of cross-validation or adaptive resampling. A systematic investigation of these strategies lies beyond the scope of this paper, but represents a natural direction for future research.

\subsection{Comparison with state of the art}\label{subsec:SOTA}
As a final numerical experiment, we test the computational efficiency of our techniques, comparing their performance to that of two state-of-the-art methods for surrogate modeling of improper transfer functions, namely the ``improper Loewner framework'' from \cite{morAntGH20} and the ``polynomial AA'' approach from \cite{GosH23}. We run all methods on the Oseen test case with $\mathsf{Re}=20$, using the same $100$ input samples, at frequencies logarithmically spaced between $10^{-1}$ and $10^1$.

It is crucial to highlight that both considered state-of-the-art methods only perform transfer function approximation (surrogate modeling), without inferring the degree of the transfer function itself. Rather, they assume a known, fixed relative degree, notably, equal to $1$. Since the state-of-the-art methods do not estimate the relative degree, we compare our full pipeline (including degree identification) against their surrogate-building step alone.

\begin{table}[htbp]
\caption{Summary of runtime comparison (averaged over three runs) for the Oseen test case ($\mathsf{Re}=20$) using different methods. The tests were run in MATLAB\textsuperscript{\textregistered} R2024b on a workstation with an Intel\textsuperscript{\textregistered} Core\textsuperscript{\texttrademark} Ultra 7 processor 265. All times are in seconds.}
\label{tab:times}
\begin{center}
	\begin{tabular}{c|c|c|c|c|}
		\multirow{2}{*}{Method} & \multicolumn{2}{c|}{Sampling time} & Postprocessing\\
		\cline{2-3}
		& base (low-freq.) & extra (hi-freq.) & time\\
		\hline
		Improper Loewner \cite{morAntGH20} & \multirow{4}{*}{\texttt{9.42e+1}} & \texttt{7.78e+0} & \texttt{6.60e-3}\\
		\cline{1-1}\cline{3-4}
		Polynomial AA \cite{GosH23} & & - & \texttt{2.47e-3}\\
		\cline{1-1}\cline{3-4}
		Auto-degree AAA [this paper] & & - & \texttt{8.77e-3}\\
		\cline{1-1}\cline{3-4}
		Auto-degree VF [this paper] & & - & \texttt{1.62e-2}\\
		\hline
	\end{tabular}
\end{center}
\end{table}

Although this makes the comparison skewed against our methods, the results, in Table~\ref{tab:times}, show that the runtimes of all algorithms are comparable. Specifically, the sampling costs are larger than the postprocessing rational-approximation costs by about $4$ orders of magnitude. This should give a representative view of the typical runtime split when our method is applied to large-scale applications, where sampling is expensive due to the system's size.

We also note that the ``improper Loewner'' method \cite{morAntGH20} requires some extra samples to approximate the improper (polynomial) part of the transfer function. In our tests, we augment the base (low-frequency) sample set with an additional $10\%$ high-frequency samples. This leads to a higher sampling cost compared to other methods.

Moving now to comparing postprocessing costs, ``polynomial AA'' is the fastest method, running about $3.5$ times quicker than our AAA-based method. This is ultimately because our method must build four different AAA surrogates within its iterative degree-identification routine, cf.~Algorithm~\ref{algo}, while no degree identification is performed in ``polynomial AA''. On the other hand, we see that incorporating VF in our method roughly doubles the postprocessing runtime, compared to the more efficient AAA algorithm.

These results indicate that, even when degree identification is included, the total cost of our pipeline is comparable to existing specialized methods, while providing strictly more information (the relative degree).

\section{Conclusions}\label{sec:conclusions}

We have presented several theoretical results concerning the barycentric rational form, primarily focusing on the concept of ``relative degree''. Based on such results, we have described two novel numerical algorithms for data-driven system identification.

The first one, Algorithm~\ref{algo:aaa}, performs rational approximation while enforcing a prescribed relative degree. Notably, in contrast to other methods, we are able to achieve non-trivial relative degrees without making any modifications to the barycentric rational form, no matter how large the enforced degree.

The second one, Algorithm~\ref{algo}, performs a similar task, but it is the algorithm itself that supplies a guess for the target system's relative degree. This is achieved through model selection by leveraging the trade-off between a rational function's complexity and degree. In contrast to competitor strategies, this degree identification is carried out without the need for high-frequency samples. This latter method has the great advantage of not requiring the user to provide the target system's relative degree as input (or a guess thereof). As such, our method can enable degree identifications in settings where no effective alternatives exist, e.g., because an accurate approximation of the transfer function at large frequencies is unfeasible due to the associated sampling/experimental costs.

Our methods' capabilities, in terms of both degree identification and rational approximation, have been showcased through extensive numerical testing. In our tests, we were not able to achieve perfect degree-identification accuracy, especially for noisy data. This is to be expected, since the effectiveness of data-driven methods is ultimately dependent on the quality of the available data. Still, we believe that our results show the great promise of our method. We envision further testing on more challenging engineering applications, so as to better identify the strengths and weaknesses of our approach. A runtime comparison of our methods with some available alternatives showed that, although some additional costs have to be expected due to our iterative degree-identification strategy, the overall runtime of the technique is essentially the same as available surrogate modeling techniques that do \emph{not} incorporate any degree identification.

At the same time, we believe that it might still be possible to increase the robustness of our method (e.g., to noise in the data) through theoretical means. To this aim, one would have to devise better (heuristic or not) criteria for comparing rational functions of different degrees, compared to those discussed in Sections~\ref{sec:DegreeAutomatic} and \ref{sec:Extensions}. At the moment, this remains a future research direction.

\appendix

\section{Mass-train example}\label{ap:cars}
\begin{figure}[tbp]
\centering
\includegraphics{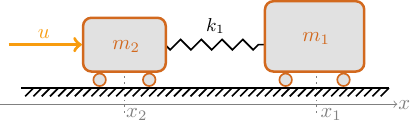}
\caption{Illustration of the sample mechanical system with $n=2$ masses.}
\label{fig_2Car_example}
\end{figure}
Consider two masses that are connected via a spring as illustrated in Figure~\ref{fig_2Car_example}. This mechanical system is modeled with two degrees of freedom, namely, the positions $x_1$, $x_2$. A force of intensity $u$ acts on the second mass, and the position $x_1$ serves as the output $y$. The model is given as
\begin{equation*}
\begin{cases}
	m_1 \ddot x_1 = - k (x_1 - x_2 - d),\\
	m_2 \ddot x_2 = k (x_1 - x_2 - d) + u,\\
	y = x_1.
\end{cases}
\end{equation*}
Friction is neglected. On top of the masses $m_1$ and $m_2$, two parameters are present: the spring constant $k>0$ and the spring length $d>0$. While the forward problem $u\to y$ is an integrator chain, the inverted dynamics $y\to u$ form a DAE of index 5 \cite{AltH17}. 

In our simulations, we consider the frequency-domain formulation of the generalization of this problem to a chain of $n\geq 2$ masses. We first look at the system
\begin{equation*}
\begin{cases}
	M\ddot{\mathbf{x}} = A\mathbf{x} + \mathbf{f} + Bu, \\
	y = C\mathbf{x}.
\end{cases}
\end{equation*}
Above, $\mathbf{x}=[x_1,\ldots,x_n]^\top$ is the vector of the masses' coordinates, the mass values are merged into $M=\textup{diag}(m_1,\ldots,m_n)\in\mathbb R^{n\times n}$, while
\begin{equation*}
A = {
	\begin{bmatrix}
		-k_1 & k_1 & \\
		k_1 & -k_1-k_2 & k_2 \\
		& \ddots & \ddots & \ddots \\
		&				&k_{n-2} &- k_{n-2} -k_{n-1}&k_{n-1} \\
		&				& &  k_{n-1}&-k_{n-1} 
\end{bmatrix}},\;\mathbf{f} = {
	\begin{bmatrix} 
		k_1d_1 \\ -k_1d_1 +k_2d_2 \\ \vdots \\ -k_{n-2}d_{n-2} +k_{n-1}d_{n-1} \\ -k_{n-1}d_{n-1}
\end{bmatrix}},
\end{equation*}
$B=[0,0,\ldots,0,1]^\top\in\mathbb R^{n\times 1}$, and $C=[1,0,\ldots,0,0]\in \mathbb R^{1\times n}$. The spring constants $k_1,\dots,k_{n-1}>0$ and spring lengths $d_1,\dots,d_{n-1}>0$ are arbitrary, although, in our numerical tests, we set all constants to $1$: $m_j=k_j=d_j=1$ for all $j$.

In the frequency domain, the $u\to y$ map is obtained as an affine linear (with respect to $u$) map
\begin{equation}\label{eq:carforw}
y(s) = C(s^2M-A)^{-1}\left(Bu(s) + \frac 1s \mathbf{f}\right)
\end{equation}
and the inverted map $y\to u$ as
\begin{equation}\label{eq:carinv}
u(s) = \left(C(s^2M-A)^{-1}B\right)^{-1}\left(y(s) - C(s^2M-A)^{-1}\frac 1s
\mathbf{f}\right).
\end{equation}
As the affine part can be estimated separately, in the numerical tests, we will only consider the linear part of the transfer function. Otherwise stated, we set $d_1=\ldots=d_{n-1}=0$, so that $\mathbf{f}=\mathbf{0}$ in \eqref{eq:carforw} and \eqref{eq:carinv}.

\section*{Declarations}
\paragraph{Data availability.} 
The MATLAB code and data files for reproducing the presented experiments are available publicly at~\cite{CodZen24} under the MIT license.

\paragraph{Conflict of interest.}
The authors declare that they have no conflict of interest.

\bibliographystyle{plainurl}
\bibliography{mybib}

\end{document}